\documentclass[14pt,hyp,reqno,]{amsart}
\usepackage{hyperref}
\usepackage{enumerate}
\hypersetup{nesting=true,debug=true,naturalnames=true}
\usepackage{graphicx,amssymb,upref}
\usepackage[utf8]{inputenc}
\usepackage{tikz-cd}
\usepackage{mathtools}
\usepackage{subcaption} 
\usepackage{mxedruli}
\usepackage{mathabx}

\usepackage{fancyhdr}
\fancypagestyle{firststyle}
{
    \setlength{\headheight}{25pt}
    \setlength{\headsep}{15pt}
}

\newcommand{\id}{ \mathrm{id}}

\newcommand{\Aut}{\mathrm{Aut}}

\renewcommand{\epsilon}{\varepsilon}

\newcommand{\e}{\mathrm{e}}

\DeclareFontFamily{OT1}{pzc}{}
\DeclareFontShape{OT1}{pzc}{m}{it}{<-> s * [1.1] pzcmi7t}{}
\DeclareMathAlphabet{\mathpzc}{OT1}{pzc}{m}{it}

\let\<\langle
\let\>\rangle

\newcommand{\norm}[1]{\left\lVert#1\right\rVert}

\subjclass[2010]{}

\DeclareRobustCommand{\oni}{%
 \daleth
}
\DeclareRobustCommand{\ghani}{%
 \beth
}

\newtheorem{theorem}{Theorem}[section]

\newtheorem{proposition}[theorem]{Proposition}
\newtheorem{conjecture}[theorem]{Conjecture}
\newtheorem{corollary}[theorem]{Corollary}
\newtheorem{lemma}[theorem]{Lemma}

\theoremstyle{definition}

\newtheorem{definition}[theorem]{Definition}

\date{\today}
\begin{document} 

\lhead{CCR and CAR algebras are connected via a path of Cuntz-Toeplitz algebras}
\rhead{}
\author{
  Alexey Kuzmin
  \\
}
\title{CCR and CAR algebras are connected via a path of Cuntz-Toeplitz algebras}
\maketitle

\begin{abstract}
For $q \in \mathbb{R}$, $|q| < 1$ we consider the universal enveloping $C^*$-algebra of a $*$-algebra of $q$-canonical commutation relations ($q$-CCR), which is generated by $a_1, \ldots, a_n$ subject to the relations
\[ a_i^* a_j = \delta_{ij} 1 + q a_j a_i^* . \]
It has a distinguished representation $\pi_F$ called the Fock representation, which is believed to be faithful. In this article we denote the image of the universal enveloping $C^*$-algebra of $q$-CCR in the Fock representation by $\ghani_q$. The question whether $C^*$-isomorphism $\ghani_q \simeq \ghani_0$ holds has been considered in the literature and proved for $|q| < 0.44$. In this article we show that $\ghani_q \simeq \ghani_0$ for $|q| < 1$.
\end{abstract}
\tableofcontents

\section{Introduction}
A broad class of operator algebras studied in the literature are universal enveloping $C^*$-algebras of $*$-algebras defined by polynomial relations:
\[ \mathbb{C}(V(p)) = \mathbb{C}[x_1,\ldots,x_n,x_1^*,\ldots,x_n^*] / (p_\alpha(x_1,\ldots,x_n,x_1^*,\ldots,x_n^*) = 0)_{\alpha \in I}, \]
where the variables do not commute. Such $*$-algebras are objects of study of noncommutative algebraic geometry. If $\mathbb{C}(V(p))$ has bounded $*$-representations on a Hilbert space, one can consider the universal enveloping $C^*$-algebra of "continuous functions on the noncommutative algebraic variety $V(p)$": \[C(V(p)) = C^*(\mathbb{C}(V(p))). \] This $C^*$-algebra encodes "topology" of the corresponding noncommutative variety. Unfortunately, up to a restricted knowledge of the author, there has been not so many examples of interplay between noncommutative algebraic geometry and operator algebras. Thus the author wishes to highlight the possibility of such point of view on $C^*$-algebras generated by generators and relations. In this paper we will consider a special class of "quadratic noncommutative curves" called Wick algebras. We will be interested in the "topological uniformization" of such noncommutative quadrics, i.e. in classification of $C(V(p))$ up to $C^*$-isomorphism for a certain class of quadratic noncommutative polynomials $p$.

Wick algebras have originated not from algebraic geometry, but from the study of non-classical models of mathematical physics, quantum group theory and noncommutative probability (see e.g., \cite{BoSpe2,fiv,green,mac,MPe,zag,Dalet,jeu_pinto,jps2,LiM,yakym,Popescu,prolett}), which gave rise to a number of papers on operator algebras generated by various deformed commutation relations \cite{BoSpe,Klimek,mar}, which are $C^*$-algebras of noncommutative algebraic varieties. They include deformations of canonical commutation relations of quantum mechanics, some quantum groups and quantum homogeneous spaces, see e.g., \cite{gisselson,Klimyk,vaksman1,woronowicz}. 

Also Wick algebras can be considered as deformations of Cuntz-Toeplitz algebras, see \cite{cun,dn,jsw,MR1291240,MR1918355,MR1893475}. This point of view will be considered in this article. Below we give more details.

Let us formally define Wick algebras. They were defined in \cite{jsw}.
For $\{T_{ij}^{kl},\ i,j,k,l=\overline{1,n}\}\subset\mathbb C$, $T_{ij}^{kl}=\overline{T}_{ji}^{lk}$, the Wick algebra  $W(T)$
is the $*$-algebra generated by elements $a_j$, $a_j^*$, $j=\overline{1,n}$ subject to the relations
\[
a_i^*a_j=\delta_{ij}\mathbf 1+\sum_{k,l=1}^n T_{ij}^{kl} a_l a_k^*.
\]
It depends \cite{jsw}  on the so called operator of coefficients $T$, 
 given as follows. Let $\mathcal{H}=\mathbb C^n$ and $e_1$, \dots, $e_n$ be the standard orthonormal basis, then
\[
T\colon \mathcal{H}^{\otimes 2}\rightarrow \mathcal{H}^{\otimes 2},\quad
T (e_k\otimes e_l)=\sum_{i,j=1}^d T_{ik}^{lj} e_i\otimes e_j.
\]
It is a non-trivial and central problem in the theory of Wick algebras to determine whether a Fock representation $\pi_{F}$ of a Wick algebra exists, see \cite{BoSpe,jps,jsw}. Fock representation is determined uniquely up to a unitary equivalence by the following property: there exists a cyclic vector $\Omega$ such that $ \pi_F(a_i^*)(\Omega) = 0$ for $i = 1, \ldots, n $. The problem of existence and uniqueness of $\pi_F$ was studied in \cite{fiv}, \cite{Bozejko}, \cite{Zagier_positivity} and in \cite{jsw} for a more general class of Wick algebras. For some sufficient conditions it exists, for example if $T$ is braided, i.e., $(\mathbf 1\otimes T)(T\otimes\mathbf 1)
(\mathbf 1\otimes T)=(T\otimes\mathbf 1)
(\mathbf 1\otimes T)(\mathbf 1\otimes T)$, and if $\lVert T\rVert\leq 1$; moreover, if $\lVert T\rVert<1$ then $\pi_{F}$ is a bounded faithful representation of $W(T)$.

Another important question concerns the stability of isomorphism classes of the universal $C^*$-envelope
$\mathcal{W}(T) = C^*(W(T))$. It was conjectured in  \cite{jsw2}:
\begin{conjecture}\label{q_stab}
If  $T$ is  self-adjoint, braided and $||T|| < 1$, then $\mathcal{W}(T) \simeq \mathcal{W}(0)$.
\end{conjecture}

In particular, the authors of \cite{jsw2} have shown that the conjecture holds for the case $||T|| < \sqrt{2} - 1$, for more results on the subject see \cite{dn}, \cite{nk}.

In the case $T=0$ and $n=\dim \mathcal{H} =1$, the Wick algebra $W(0)$ is generated by a single isometry $s$, its  universal $C^*$-algebra exists and is isomorphic to the $C^*$-algebra generated by the unilateral shift, and the Fock representation is faithful. The ideal $\mathcal I$ in $\mathcal E$, generated by $\mathbf 1 - ss^*$ is isomorphic to the algebra of compact operators and $\mathcal E/\mathcal I\simeq C(S^1)$, see \cite{coburn}. When $n\ge 2$, the enveloping universal $C^*$-algebra   exists and it is called the Cuntz-Toeplitz agebra $\mathbb{K}\mathcal{O}_n$. It is isomorphic to $C^*(\pi_{F}(W(0)))$, so the Fock representation of $\mathbb{K}\mathcal{O}_n$ is faithful, see \cite{cun}. Furthermore, the ideal generated by $1-\sum_{j=1}^n s_js_j^*$ is the unique largest ideal in $\mathbb{K}\mathcal{O}_n$. It is isomorphic to the algebra of compact operators on $\mathcal F_n$. The quotient $\mathbb{K}\mathcal{O}_n/\mathbb{K}$ is called the Cuntz algebra $\mathcal O_n$. It is nuclear (as well as $\mathbb{K}\mathcal{O}_n$),  simple and purely infinite, see \cite{cun} for more details.

Among Wick algebras, considerable attention has been paid to the study of so-called $q$-CCR introduced by M. Bozejko and R. Speicher, see \cite{Bozejko} which as a Wick algebra corresponds to the operator $T(x \otimes y) = q y \otimes x$. Assume that $q \in \mathbb{R}$, $|q| < 1$ . Define $q$-CCR to be a $*$-algebra generated by elements $a_i, a_i^*$, $i = 1, \ldots, n$, satisfying the following relations:
\[ a_i^* a_j = \delta_{ij} + qa_j a_i^*. \]
It is a deformation of $*$-algebras of the classical commutation relations in the sense that in the Fock realisation, the limiting cases $q = 1$ and $q = -1$ correspond to $*$-algebras of the canonical commutation relations (CCR) and the canonical anti-commutation relations (CAR) respectively.

It can easily be verified that in any $*$-representation $\pi$ of the $*$-algebra $q$-CCR by bounded operators one has
\[ \norm{\pi(a_i)} \leq \frac{1}{\sqrt{1-|q|}}, \text{ } i = 1, \ldots, n. \]
Hence, there exists a universal enveloping $C^*$-algebra associated to $q$-CCR. We denote it's image in the Fock representation by $\ghani_q$. See more about the Fock representation in \cite{BoLyt}.

Let us formulate the main result of this article:
\begin{theorem}
Let $|q| < 1$, then
\[ \ghani_q \simeq \ghani_0 \simeq \mathbb{K}\mathcal{O}_n. \]
\end{theorem}

Returning to the point of view of noncommutative algebraic geometry, this result can metaphorically be considered as a topological uniformization of noncommutative quadrics. For commutative quadrics there are three families: parabolas, hyperbolas and ellipses. For the subclass of noncommutative quadrics given by $q$-CCR, we therefore have either "elliptic" Cuntz-Toeplitz algebra, "parabolic" fermionic algebra and "hyperbolic" bosonic algebra. 

\thispagestyle{firststyle}

Many authors were interested in the study of the $C^*$-algebra generated by operators of the Fock representation of $q$-CCR. Namely, K. Dykema and A. Nica in \cite{dn} proved that $\ghani_q$ $ \simeq \mathbb{K}\mathcal{O}_n$ for $|q| < 0.44$ which is slightly larger than $\sqrt{2} - 1$ - result of \cite{jsw}. Also an embedding of $\mathbb{K}\mathcal{O}_n$ into $\ghani_q$ was constructed for $|q| < 1$.

Later M. Kennedy in \cite{nk} showed existence of an embedding of $\ghani_q$ into $\mathbb{K}\mathcal{O}_n$ and proved that $\ghani_q$ is an exact $C^*$-algebra.

Let us stress out that results concerning $\ghani_q$ cannot be automatically lifted to the universal $C^*$-algebra since at the moment we do not know whether or not $\pi_F$ is a faithful $*$-representation of $C^*(q$-CCR) for any $|q| < 1$. However $\pi_F$ is a faithful representation of $q$-$CCR$, i.e. it is faithful on the $*$-algebraic level.

\section{Setup} 
Let $q \in \mathbb{R}$, $|q| < 1$ and $n \in \mathbb{N}$. In this section we will define $C^*$-algebras to be considered in the article, the Fock representation, auxiliary operators, state and prove basic structural facts about the algebras.

\begin{definition}[$q$-deformed Fock space]
Let $\mathcal{F} = \bigoplus_{k = 0}^\infty (\mathbb{C}^n)^{\otimes k}$ be a linear space. Endow it with the following inner product, see \cite{Bozejko}, \cite{dn}
\[ \langle \xi_1 \otimes \ldots \otimes \xi_k, \eta_1 \otimes \ldots \otimes \eta_k \rangle_q = \sum_{\sigma \in S_k} q^{\mathsf{inv}(\sigma)} \langle \xi_{\sigma_1}, \eta_1 \rangle \ldots \langle \xi_{\sigma_k}, \eta_k \rangle. \]
The pair $(\mathcal{F}^q, \langle \cdot, \cdot \rangle)$ is called the $q$-deformed Fock space. We denote $\mathcal{F}^q_k$ to be the $k$-th component of $\mathcal{F}^q$. Put $\Omega$ to be the unit vector in $\mathcal{F}^q_0$ and call it vacuum vector.
\end{definition}

\begin{definition}[Creation operators]
Let $e_1, \ldots, e_n$ be an orthonormal basis in $\mathbb{C}^n$. Define
\[L_i^q : \mathcal{F}^q \rightarrow \mathcal{F}^q, \ L_i^q(\xi) := e_i \otimes \xi. \]
\[R_i^q : \mathcal{F}^q \rightarrow \mathcal{F}^q, \ R_i^q(\xi) := \xi \otimes e_i. \]
These operators are called correspondingly left and right creation operators
\end{definition}

\begin{definition}[Annihilation operators]
With respect to the $q$-inner product on the $q$-deformed Fock space the adjoints to left and right creation operators have the following form
\[(L_i^q)^* : \mathcal{F}^q \rightarrow \mathcal{F}^q, (L_i^q)^*(e_{i_1} \otimes \ldots \otimes e_{i_k}) = \sum_{m = 1}^k q^{m - 1} \delta_{i i_m} e_{i_1} \otimes \ldots \otimes \widehat{e_{i_m}} \otimes \ldots \otimes e_{i_k},
\]
\[(R_i^q)^* : \mathcal{F}^q \rightarrow \mathcal{F}^q, (R_i^q)^*(e_{i_1} \otimes \ldots \otimes e_{i_k}) = \sum_{m = 1}^k q^{k - m} \delta_{i i_m} e_{i_1} \otimes \ldots \otimes \widehat{e_{i_m}} \otimes \ldots \otimes e_{i_k}, \]
where by $\widehat{\cdot}$ we mean that the tensor is missed. We call $(L_i^q)^*$, $(R_i^q)^*$ left and right annihilation operators.
\end{definition}

\begin{definition}[Fock representation]
Define $\pi_F^L$ and $\pi_F^R$ to be the left and right Fock representations of the $*$-algebra of $q$-CCR and define them to be
\[ \pi_F^L(a_i) = L_i^q, \ \pi_F^R(a_i) = R_i^q, \ i = 1,\ldots, n. \]
\end{definition}

\begin{definition}[$C^*$-algebra of the $q$-CCR in the Fock representation]
We call $\ghani^L_q$, $\ghani^R_q$ to be the image of the universal enveloping $C^*$-algebra of $q$-CCR in $\pi_F^L$, $\pi_F^R$ respectively. In other words,
\[ \ghani^L_q = C^*(L_1^q, \ldots, L_n^q), \]
\[ \ghani^R_q = C^*(R_1^q, \ldots, R_n^q). \]
We will also use a compound version
\[ \ghani^{L,R}_q = C^*(L_1^q, \ldots, L_n^q, R_1^q, \ldots, R_n^q). \]
\end{definition}

\begin{definition}[Tensor-reverse operator]
The following operator on $\mathcal{F}^q$ we call tensor-reverse operator
\[J^q : \mathcal{F}^q \rightarrow \mathcal{F}^q, \ J^q(\xi_1 \otimes \ldots \otimes \xi_k) := \xi_k \otimes \ldots \otimes \xi_1. \]
It can be seen that $J^q$ is unitary, see \cite{nk}.
\end{definition}

It is easy to check that $C^*$-algebras $\ghani^L_q$ and $\ghani^R_q$ are $C^*$-isomorphic with isomorphism given by $\mathsf{Ad}(J^q)$. Indeed, $J^q L_i^q J^q = R_i^q$. Thus when there is no need to distinguish $\ghani^L_q$ and $\ghani^R_q$ we will simply write $\ghani_q$. Another consequence of the fact that $J^q L_i^q J^q = R_i^q$ is that $\mathsf{Ad}(J^q)$ is an automorphism of $\ghani_q^{L,R}$.

\begin{definition}[Particle number operator]
We define operators called particle number operator
\[ \rho_L = \sum_{i = 1}^n L_i^q (L_i^q)^*, \ \rho_R = \sum_{i = 1}^n R_i^q (R_i^q)^*. \]
Their action on tensors are given by
\[ \rho_L(\e_{i_1} \otimes \ldots \otimes \e_{i_m}) = \sum_{k = 1}^n q^{k-1} \e_{i_1} \otimes \ldots \widehat{e_{i_k}} \ldots \otimes e_{i_m}. \]
\[ \rho_R(\e_{i_1} \otimes \ldots \otimes \e_{i_m}) = \sum_{k = 1}^n q^{m-k} \e_{i_1} \otimes \ldots \widehat{e_{i_k}} \ldots \otimes e_{i_m}. \]
\end{definition}

\begin{definition}[Gauge action on $\ghani_q^L$, $\ghani_q^R$]
Let $z \in \mathbb{T}$. Consider operators $(L_i^q)' = z L_i^q$. Then $(L_1^q)', \ldots, (L_n^q)'$ satisfy $q$-commutation relations and $((L_i^q)')^*(\Omega) = 0$, so by the uniqueness of the Fock representation, there exists a unitary $U_z$ which intertwines $(L_i^q)'$ and $L_i^q$. The same unitary intertwines $R_i^q$ and $z R_i^q$. Thus we can define the following action $\gamma$ of the torus $\mathbb{T}$ on $\ghani_q^{L, R}$:
\[ \gamma_{z}(x) = U_z x U_z^*, \]
which acts on generators by
\[ \gamma_z(L_i^q) = z \cdot L_i^q, \ \gamma_{z}(R_i^q) = z \cdot R_i^q. \]
The action $\gamma$ induces actions of $\mathbb{T}$ on $\ghani_L^q$ and $\ghani_R^q$ by
\[ z \curvearrowright L_i^q = \gamma_{z} (L_i^q), \ z \curvearrowright R_i^q = \gamma_{z} (R_i^q). \]
\end{definition}

\begin{definition}[Fixed point $C^*$-subalgebra]
We denote
\[ \ghani_q^\mathbb{T} = \{ x \in \ghani_q : z \curvearrowright x = x, \ z \in \mathbb{T} \}, \]
\[ (\ghani_q^{L, R})^{\mathbb{T}} = \{ x \in \ghani_q^{L, R} : z \curvearrowright x = x, \ z \in \mathbb{T} \}. \]
\end{definition}

\begin{proposition}[Orthogonal projections onto the vacuum vector]
The orthogonal projection $P_{\Omega}$ onto $\mathcal{F}^q_0$ belong both to $\ghani_q^L$ and $\ghani_q^R$.
\end{proposition}
\begin{proof}
By Lemma 4.1 of \cite{dn}, $\ker \rho_L = \ker \rho_R =\mathcal F_0^q=\left<\Omega\right>$ and there exist $C_1,\, C_2>0$ independent of $m > 0$, such that
\begin{equation}\label{rho_estimate1}
C_1 1_{\ghani_q^L} < (\rho_L)_{|\mathcal F_m^q} < C_2 1_{\ghani_q^L},\quad m\in\mathbb N
\end{equation}
\begin{equation}\label{rho_estimate2}
C_1 1_{\ghani_q^R} < (\rho_R)_{|\mathcal F_m^q} < C_2 1_{\ghani_q^R},\quad m\in\mathbb N
\end{equation}
In particular (\ref{rho_estimate1},\ref{rho_estimate2}) implies that $0$ is an isolated point in the spectrum of $\rho_L$ and $\rho_R$, hence the spectral projection 
$E_{\rho_L}(0)=E_{\rho_R}(0)=P_{\mathcal F_0^q}$ is contained in $\ghani_q^L$ and $\ghani_q^R$.

\end{proof}

\begin{proposition}[Invariance of the ideal of compact operators]\label{Invariance}
The ideal of compact operators $\mathbb{K}(\mathcal{F}^q)$ is contained both in $\ghani_q^L$ and $\ghani_q^R$. Moreover, it is invariant under the action of $\mathbb{T}$.
\end{proposition}
\begin{proof}
The orthogonal projection $P_\Sigma$ is a compact operator and belongs to both $\ghani_q^L$ and $\ghani_q^R$. Both $\ghani_q^L$ and $\ghani_q^R$ are irreducible $C^*$-subalgebras of $\mathbb{B}(\mathcal{F}^q)$, so by Corollary I.10.4 of \cite{dav}, the whole ideal of compact operators is contained in $\ghani_q^L$, $\ghani_q^R$. 

Since the action of $\mathbb{T}$ is implemented by conjugation with a unitary, $\mathbb{K}(\mathcal{F}^q)$ is invariant being an ideal in both $\oni_q^L$ and $\oni_q^R$.
\end{proof}

\begin{definition}[Quotient of $q$-CCR]
We denote $\oni_q^L = \ghani_q^L / \mathbb{K}(\mathcal{F}^q)$, $\oni_q^R = \ghani_q^R / \mathbb{K}(\mathcal{F}^q)$, $\oni_q^{L,R} = \ghani_q^{L,R} / \mathbb{K}(\mathcal{F}^q)$. When there is no need to distinguish $\oni_q^L$ and $\oni_q^R$ we will simply write $\oni_q$.
\end{definition}

\begin{definition}[Gauge action on $\oni_q^L$, $\oni_q^R$]
Since the ideal of compact operators is $\mathbb{T}$-invariant, the action of $\mathbb{T}$ descends to $\oni_q^L$ and $\oni_q^R$.
\end{definition}

\begin{proposition}[Theorem 4.3 of \cite{dn}]
There is a $\mathbb{T}$-equivariant inclusion $\mathbb{K}\mathcal{O}_n \subset \ghani_q^L$ and $\mathbb{K}\mathcal{O}_n \subset \ghani_q^R$. Moreover, these inclusions are implemented by conjugation with the same unitary $U : \mathcal{F}^0 \rightarrow \mathcal{F}^q$.
\end{proposition}
\begin{proposition}[Theorem of \cite{nk}]
There is an inclusion $\ghani_q^L \subset \mathbb{K}\mathcal{O}_n$ and $\ghani_q^R \subset \mathbb{K}\mathcal{O}_n$. Moreover, these inclusions are implemented by conjugation with the same unitary $U^{opp} : \mathcal{F}^q \rightarrow \mathcal{F}^0$.
\end{proposition}

These propositions are reformulations of the corresponding Theorems in \cite{dn} and \cite{nk}. $\mathbb{T}$-equivariance follows from the fact that $U$ conjugates generators $s_1, \ldots s_n$ of $\mathbb{K}\mathcal{O}_n$ into operators of the form $R a_1, \ldots, R a_n$ with $R$ being $\mathbb{T}$-equivariant and $a_1, \ldots, a_n$ generators of $\ghani_q$. 

In both cases the ideal of compact operators is mapped into the ideal of compact operators because the inclusion is given by conjugation with a unitary. Moreover, since the quotient map $\ghani_q \to \oni_q$ is $\mathbb{T}$-equivariant, we can conclude the following Propositions
\begin{corollary}
There is a $\mathbb{T}$-equivariant inclusion $\mathcal{O}_n \subset \oni_q$.
\end{corollary}
\begin{corollary}
There is an inclusion $\oni_q \subset \mathcal{O}_n$
\end{corollary}

\section{Description of the approach}
\subsection{Step 1, approximation of flip: $\oni_q^{\mathbb{T}} \simeq \oni_0^{\mathbb{T}}$}
\begin{definition}
Let $A$ be a $C^*$-algebra. Consider an automorphism $\sigma$ of $A \otimes_{min} A$ given by $\sigma(a \otimes b) = b \otimes a$. $A$ has approximately inner flip if there is a sequence $u_1, u_2, \ldots$ of unitaries in $A \otimes_{min} A$ such that for each $x \in A \otimes_{min} A$
\[ \lim_{k \to \infty} \lVert u_k x u_k^* - \sigma(x) \rVert = 0. \]
\end{definition}
Let $\mathfrak{U}_{n^\infty}$ be the uniformly hyperfinite algebra of type $n^\infty$.
\begin{theorem}[\cite{RosenbergEffros}, Theorem 5.1]\label{RosenbergUHF}
Let $A$ be a unital $C^*$-algebra. Then $A \simeq \mathfrak{U}_{n^\infty}$ iff
\begin{enumerate}
    \item $A$ has approximately inner flip.
    \item $A$ has an asymptotic imbedding in $\mathfrak{U}_{n^\infty}$.
    \item $A \simeq A \otimes \mathfrak{U}_{n^\infty}$.
\end{enumerate}
\end{theorem}
In order to prove that $\oni_q^\mathbb{T} \simeq \oni_0^\mathbb{T}$ we show that $\oni_q^\mathbb{T}$ satisfies conditions of Theorem \ref{RosenbergUHF}, thus $\oni_q^\mathbb{T} \simeq \mathfrak{U}_{n^\infty}$.

\subsection{Step 2, crossed product by an endomorphism: $\oni_q \simeq \oni_q^{\mathbb{T}} \rtimes \mathbb{N}$}
There are different models for a crossed product by endomorphism: Doplicher, Stacey, Murphy, Exel, Paschke, Kwasniewski gave their different visions on it. Under certain conditions these models coincide.

\begin{definition}
Let $A$ be a $C^*$-algebra and $S$ be a nonunitary isometry in $\mathbb{B}(\mathcal{H})$ such that $S A S^* \subset A$ and $S^* A S \subset A$. Then $C^*(A, S)$ is called crossed product of $A$ by $S$ in the sense of Paschke.
\end{definition}

We will need to use results for the Stacey crossed product by endomorphism, so we state a result which ensures that the Paschke crossed product is isomorphic to the Stacey crossed product.

\begin{proposition}[\cite{Ortega}, Example 1.19]\label{OrtegaEquivalenceModels}
Let $A$ be a $C^*$-algebra with faithful trace and $S$ be a nonunitary isometry such that $S A S^* \subset A$ and $S^* A S \subset A$. Suppose there are no nontrivial ideals $I$ in $A$ such that $S I S^* \subset I$ or $S^* I S \subset I$. Define endomorphism $\beta$ of $A$ to be $\beta(a) = S a S^*$. Then the Stacey crossed product $A \rtimes_\beta \mathbb{N}$ is isomorphic to the Paschke crossed product $C^*(A, S)$. Moreover, the crossed product is a simple $C^*$-algebra.
\end{proposition}

Using Proposition \ref{OrtegaEquivalenceModels} we prove that $\oni_q$ is isomorphic to the Stacey crossed product of $\oni_q^\mathbb{T}$ by an endomorphism.

\subsection{Step 3, Kirchberg-Philips classification: $\oni_q \simeq \oni_0$}
\begin{theorem}[\cite{Ortega}, Theorem 3.6]\label{OrtegaPF}
Let $A$ be a unital non-type I $C^*$-algebra of real rank zero that has strict
comparison, let $\beta$ be an injective endomorphism of $A$ such that $\beta(1) \neq 1$ and $\beta(A)$ is a
hereditary sub-$C^*$-algebra of $A$. If the Stacey crossed product $A \rtimes_\beta \mathbb{N}$ is simple and $\beta(1)$ is a full projection of $A$, then $A \rtimes_\beta \mathbb{N}$ is purely infinite simple $C^*$-algebra.
\end{theorem}

We use Theorem \ref{OrtegaPF} to show that $\oni_q$ is nuclear purely infinite simple $C^*$-algebra which satisfy Universal Coefficient Theorem.

\begin{theorem}[\cite{Philips}, Theorem 4.2.4]\label{Philips}
Let $A$ and $B$ be separable nuclear unital purely infinite simple $C^*$-algebras which satisfy Universal Coefficient Theorem, and suppose that there exists a graded isomorphism $\alpha : K_*(A) \rightarrow K_*(B)$ such that $\alpha([1_A]) = [1_B]$. Then $A \simeq B$.
\end{theorem}

We compute K-theory of $\oni_q$ and use Theorem \ref{Philips} to show that $\oni_q \simeq \oni_0$.

\subsection{Step 4, Gabe-Ruiz classification of unital extensions: $\ghani_q \simeq \ghani_0$}
Consider two six-term exact sequences
\[
x^i :
\begin{tikzcd}
H_0^i \arrow[r] & L_0^i \arrow[r] & G_0^i \arrow[d] \\
H_1^i \arrow[u] & L_1^i \arrow[l] & G_1^i \arrow[l]
\end{tikzcd}
\]
with distinguished elements $x_i \in L_0^i$, $y_i \in G_0^i$ for $i = 1, 2$. A homomorphism $(\psi_*, \rho_*, \phi_*) : x^1 \rightarrow x^2$ of six-term exact sequences consists of 
\[ \psi_* : H_*^1 \rightarrow H_*^2, \ \rho_* : L_*^1 \rightarrow L_*^2, \]
making the diagram commute and such that $\rho_0(x_1) = x_2$, $\phi_0(y_1) = y_2$. if $H_*^1 = H_*^2 = H_*$ and $G_*^1 = G_*^2 = G_*$ then we say that $x^1$ and $x^2$ are congruent if there exists isomorphism of the form $(id_{H_*}, \rho_*, id_{G_*})$.

\begin{definition}
For an extension
\[ \mathcal{E} : 0 \rightarrow B \rightarrow E \rightarrow A \rightarrow 0 \]
of unital $C^*$-algebras we let $K_{six}^u(\mathcal{E})$ denote the six-term exact sequence in K-theory with distinguished elements $[1_E] \in K_0(E)$ and $1_A \in K_0(A)$.
\end{definition}

\begin{theorem}[Proposition 5.8, \cite{GabeRuiz}]\label{Gabe}
Let 
\[ \mathcal{E}_i : 0 \rightarrow B \rightarrow E_i \rightarrow A \rightarrow 0 \]
be unital extensions of $C^*$-algebras for $i = 1, 2$ such that $A$ is a unital UCT Kirchberg algebra and $B$ is a stable AF algebra. If $K_{six}^u(\mathcal{E}_1)$ is congruent to $K_{six}^u(\mathcal{E}_2)$ then $E_1 \simeq E_2$.
\end{theorem}
Since $\ghani_q$ fits into short exact sequence
\[ 0 \rightarrow \mathbb{K} \rightarrow \ghani_q \rightarrow \oni_q \simeq \mathcal{O}_n \rightarrow 0, \]
Theorem \ref{Gabe} will be used to show that $\ghani_q \simeq \ghani_0$.

\section{Approximation of flip}
In this section we will show that $\oni_q^\mathbb{T}$ has flip approximation property. In order to comfortably make computations in $\oni_q^\mathbb{T} \otimes_{min} \oni_q^\mathbb{T}$, we will use a spatial model for a tensor product - it will be $(\oni_q^{L,R})^\mathbb{T}$.

\begin{lemma}[Lemma 3.1, \cite{Shlyakhtenko}]
\[ [(L_i^q)^*, R_j^q] \vert_{\mathcal{F}^q_n} = \delta_{ij} q^n \id_{\mathcal{F}_n^q}. \]
In particular, $[\ghani_q^L, \ghani_q^R] \subset \mathbb{K}(\mathcal{F}^q)$.
\end{lemma}

\begin{corollary} \label{Commut}
\[ [\oni_q^L, \oni_q^R] = \{0\}. \]
\end{corollary}

\begin{lemma} \label{AlgInj}
Inclusion of the algebraic tensor product $\oni_q^L \odot \oni_q^R \hookrightarrow \oni_q^{L, R}$ defined on elementary tensors by $l \otimes r \mapsto lr$ is an injective $*$-homomorphism.
\end{lemma}
\begin{proof}
The inclusion is a $*$-homomorphism because $\oni_q^L \hookrightarrow \oni_q^{L, R}$ commutes with $\oni_q^R \hookrightarrow \oni_q^{L, R}$ by Lemma \ref{Commut}.

In order to prove that the inclusion is injective, consider embedding \[ \iota : \oni_q^{L,R} \hookrightarrow \oni_0^{L, R} \subset \mathbb{B}(\mathcal{F}^0) / \mathbb{K}(\mathcal{F}^0). \]
Notice that $\oni_0^{L, R} \simeq \oni_0^L \otimes \oni_0^R$ and the inclusion $\iota$ has property $\iota(\oni_q^L) \subset \oni_0^L \otimes 1$ and $\iota(\oni_q^R) \subset 1 \otimes \oni_0^R$.

Every element of $\oni_q^L \odot \oni_q^R$ can be written as $\sum_{i = 1}^d l_i \otimes r_i$ with $l_1, \ldots, l_d$ linearly independent. Assume that $\sum_{i = 1}^d l_i r_i = 0 \in \oni_q^{L, R}$. Then 
\[ \iota(\sum_{i = 1}^d l_i r_i) = \sum_{i = 1}^d \iota(l_i) \otimes \iota(r_i) = 0 \in \oni_0^L \otimes \oni_0^R. \]
Since $\iota(l_1), \ldots, \iota(l_d)$ are also linearly independent, we conclude that $\iota(r_1) = \ldots = \iota(r_d) = 0$. By injectivity of $\iota$ we conclude $r_1 = \ldots = r_d = 0$ and $\sum_{i = 1}^d l_i \otimes r_i = 0$.  
\end{proof}

\begin{lemma}\label{TensorIso}
For some $C^*$-completion $\alpha$ on $\oni^L_q \odot \oni^R_q$,
\[ (\oni^{L,R}_q, \mathsf{Ad}(J^q)) \simeq^{\mathbb{Z}/2\mathbb{Z}} (\oni_q^L \otimes_\alpha \oni_q^R, \sigma \circ (\mathsf{Ad}(J^q) \otimes_\alpha \mathsf{Ad}(J^q))) \simeq^{\mathbb{Z}/2\mathbb{Z}} (\oni_q \otimes_\alpha \oni_q, \sigma). \]
\end{lemma}
\begin{proof}
Make $\oni_q^L \odot \oni_q^R$ into a normed space by postulating that the injection $\iota$ from Lemma \ref{AlgInj} is an isometry. This norm is a $C^*$-norm on a tensor product, thus completion of $\oni_q^L \odot \oni_q^R$ makes it into a $C^*$-algebra $\oni_q^L \otimes_\alpha \oni_q^R$. Since expressions of the form $\sum_{i = 1}^d l_i r_i$ are dense in $\oni_q^{L, R}$, $\iota : \oni_q^L \otimes_\alpha \oni_q^R \rightarrow \oni_q^{L, R}$ becomes surjective. Extensions of isometries to a completion preserve injectivity. Moreover, 
\[ \iota(\sigma(J^q l J^q \otimes J^q r J^q)) = \iota(J^q r J^q \otimes J^q l J^q) = J^q rl J^q = J^q lr J^q = \mathsf{Ad}(J^q)(\iota(l \otimes r)). \]
\end{proof}

\begin{definition}
Let $x \in \mathbb{B}(\mathcal{F}^q)$. Denote $1 \otimes x, x \otimes 1 \in \mathbb{B}(\mathcal{F}^q)$ as follows:
\[ (1 \otimes x)(\Omega) = 0, \ (x \otimes 1)(\Omega) = 0, \]
\[ (1 \otimes x)(\xi) = \xi \otimes x(\Omega), \ (x \otimes 1)(\xi) = x(\Omega) \otimes \xi, \]
\[ (1 \otimes x)(\xi_1 \otimes \ldots \otimes \xi_{n+1}) = \xi_1 \otimes x(\xi_2 \otimes \ldots \otimes x_{n+1}), \]
\[ (x \otimes 1)(\xi_1 \otimes \ldots \otimes \xi_{n+1}) = x(\xi_1 \otimes \ldots \otimes \xi_n) \otimes \xi_{n+1}. \]
\end{definition}

In what follows we will use such notation:
\begin{itemize}
    \item $\pi : \mathbb{B}(\mathcal{F}^q) \rightarrow \mathbb{B}(\mathcal{F}^q) / \mathbb{K}(\mathcal{F}^q)$.
    \item $J_k^q := 1^{\otimes k} \otimes J^q \otimes 1^{\otimes k} \in \mathbb{B}(\mathcal{F}^q)$.
    \item $U_k := J^q J_k^q = J_k^q J^q  \in \mathbb{B}(\mathcal{F}^q)$.
    \item $U_k = V_k |U_k| = |U_k|^{-1} V_k^*$ - polar decomposition.
    \item $j_k^q = \pi(J_k^q) \in \mathbb{B}(\mathcal{F}^q) / \mathbb{K}(\mathcal{F}^q)$.
    \item $u_k := \pi(U_k) \in \mathbb{B}(\mathcal{F}^q) / \mathbb{K}(\mathcal{F}^q)$.
    \item $v_k := \pi(V_k) \in \mathbb{B}(\mathcal{F}^q) / \mathbb{K}(\mathcal{F}^q)$.
\end{itemize}

\begin{lemma}
Let $x \in \mathbb{B}(\mathcal{F}^q)$. Then
\begin{itemize}
    \item $(1 \otimes x) L_i^q = L_i^q x$,
    \item $(x \otimes 1) R_i^q = R_i^q x$.
\end{itemize}
\end{lemma}
\begin{proof}
    \[ (1 \otimes x)L_i^q(\xi) = (1 \otimes x)(e_i \otimes \xi) = e_i \otimes x(\xi) = L_i^q (x(\xi)). \]
    \[ (x \otimes 1)R_i^q(\xi) = (x \otimes 1)(\xi \otimes e_i) = x(\xi) \otimes e_i = R_i^q (x(\xi)). \]
\end{proof}

\begin{lemma}\label{tensors}
If $x \in \ghani^{L,R}_q$ then $1 \otimes x \in \ghani^{L,R}_q$ and $x \otimes 1 \in \ghani^{L, R}_q$.

If $x \in \ghani_q^L$ then $1 \otimes x \in \ghani_q^L$.

If $x \in \ghani_q^R$ then $x \otimes 1 \in \ghani_q^R$.
\end{lemma}
\begin{proof}
Let $\rho_L^+$, $\rho_R^+$ be inverses of $\rho_L, \rho_R$ outside of $\mathcal{F}^q_0$. In other words, $\rho_L \rho_L^+ = \rho_L^+ \rho_L = \rho_R^+ \rho_R = \rho_R \rho_R^+ = 1 - P_\Omega$.
\[ \sum_{i = 1}^n L_i^q x (L_i^q)^* = \sum_{i = 1}^n (1 \otimes x) L_i^q (L_i^q)^* = (1 \otimes x) \rho_L, \]
\[ (1 \otimes x) = (1 \otimes x) (1 - P_\Omega) = (1 \otimes x) \rho_L \rho_L^+ = \sum_{i = 1}^n L_i^q x (L_i^q)^* \rho_L^+ \in \ghani^{L,R}_q. \]
\[ \sum_{i = 1}^n R_i^q x (R_i^q)^* = \sum_{i = 1}^n (x \otimes 1) R_i^q (R_i^q)^* = (x \otimes 1) \rho_R, \]
\[ (x \otimes 1) = (x \otimes 1) (1 - P_\Omega) = (x \otimes 1) \rho_R \rho_R^+ = \sum_{i = 1}^n R_i^q x (R_i^q)^* \rho_R^+ \in \ghani^{L,R}_q. \]
\end{proof}

\begin{lemma}
$u_k \in \oni^{L, R}_q$.
\end{lemma}
\begin{proof}
We prove that $J_k^q = J^q A_k$ for some $A_k \in \ghani^{L, R}_q$:
\begin{align*}
    J_k^q & = \sum_{i = 1}^n L_i^q (J_{k - 1}^q \otimes 1) (L_i^q)^* \rho_L^{+} = \\ & = \sum_{i = 1}^n \sum_{j = 1}^n L_i^q R_j^q J_{k - 1}^q (R_j^q)^* \rho_R^{+} (L_i^q)^* \rho_L^{+} = \\ & = \sum_{i = 1}^n \sum_{j = 1}^n L_i^q R_j^q J^q A_{k - 1} (R_j^q)^* \rho_R^+ (L_i^q)^* \rho_L^{+} = \\ & = J^q \sum_{i = 1}^n \sum_{j = 1}^n R_i^q L_j^q A_{k - 1} (R_j^q)^* \rho_R^{+} (L_i^q)^* \rho_L^{+}.
\end{align*}
Thus
\[ u_k = \pi(J^q J_k^q) = \pi(A_k) \in \oni_{L,R}^q. \]
\end{proof}

\begin{lemma}
Let $x \in \mathbb{B}(\mathcal{F}^q)$. Then for every $k$ one has \[ \lVert 1^{\otimes k} \otimes x \otimes 1^{\otimes k} \rVert < C_x. \]
\end{lemma}
\begin{proof}
\[\lim_{k \to \infty}^{\text{SOT}} (1^{\otimes k} \otimes x \otimes 1^{\otimes k}) = 0, \]
thus by Banach-Steinhaus theorem $\sup_{k \in \mathbb{N}} \lVert 1^{\otimes k} \otimes x \otimes 1^{\otimes k} \rVert < C_x$.
\end{proof}

\begin{lemma}\label{GeneratorsGhani}
$C^*(1, L_i^q (L_j^q)^*, i, j = 1\ldots n) = (\ghani^L_q)^\mathbb{T}$.
\end{lemma}
\begin{proof}
Proof is by induction. Let $|\mu| = |\sigma| = k + 1$. Then
\begin{align*}
    & L_\mu^q (L_\sigma^q)^* = L_{\mu_{1 \ldots k}}^q L_{\mu_{k+1}}^q (L_{\sigma_{1 \ldots k}}^q)^* (L_{\sigma_1}^q)^* \in \\ & \in (L_{\mu_{1 \ldots k}}^q (L_{\sigma_{1 \ldots k}}^q)^*) (L_{\mu_{k+1}}^q (L_{\sigma_1}^q)^*) + \text{span}\{ L_\alpha^q (L_\beta^q)^*, \ |\alpha| = |\beta| = k \}.
\end{align*}
\end{proof}

\begin{lemma}
For every $x \in (\oni^{L,R}_q)^\mathbb{T}$ one has $\lim_{k \to \infty} [j_k^q, x] = 0$.
\end{lemma}
\begin{proof}
Suppose $\lim_{k \to \infty} [j_k^q, x] = 0$ and $\lim_{k \to \infty} [j_k^q, y] = 0$. Then
\[ \lim_{k \to \infty} [j_k^q, xy] = \lim_{k \to \infty} x[j_k^q, y] + \lim_{k \to \infty} [j_k^q, x]y = 0, \]
\[ \lim_{k \to \infty} [j_k^q, x+y] = 0 \]

Suppose $\lim_{k \to \infty} [j_k^q, x] = 0$. Then
\[ \lim_{k \to \infty} [j_k^q, j^q x j^q] = \lim_{k \to \infty} j^q [j_k^q, x] j^q = 0, \]
\[ \lim_{k \to \infty} [j_k^q, x^*] = 0. \]
Since $(\oni_q^{L, R})^{\mathbb{T}} = C^*((\oni_q^L)^\mathbb{T}, (\oni_q^R)^{\mathbb{T}})$ and $j^q (\oni_q^L)^\mathbb{T} j^q = (\oni_q^R)^{\mathbb{T}}$, it is enough to prove the Lemma for $x$ being generators of $(\oni_q^L)^\mathbb{T}$, which by Lemma \ref{GeneratorsGhani} are $\pi(L_i^q)\pi(L_j^q)^*$.

\begin{align*}
    L_i^q (L_j^q)^* J_k^q (\xi_{i_1} \ldots \xi_{i_M}) & = \sum_{l = 1}^{k} q^{l - 1} \delta_{j, i_l} e_i \xi_{i_1} \ldots \widehat{\xi_{i_l}} \ldots \xi_{i_k} \xi_{i_{M-k}} \ldots \xi_{i_{k+1}} \xi_{i_{M-k+1}} \ldots \xi_{i_M} + \\ & +  \sum_{l = k+1}^{M-k} q^{M - l} \delta_{j, i_{l}} e_i \xi_{i_1} \ldots \xi_{i_k}  \xi_{i_{M-k}} \ldots \widehat{\xi_{i_{l}}} \ldots \xi_{i_{k+1}} \xi_{i_{M-k+1}} \ldots \xi_{i_M} + \\ & + \sum_{l = M-k+1}^{M} q^{l - 1} \delta_{j, i_l} e_i \xi_{i_1} \ldots \xi_{i_k} \xi_{i_{M-k}} \ldots \xi_{i_{k+1}} \xi_{i_{M-k+1}} \ldots \widehat{\xi_{i_l}} \ldots \xi_{i_M}.
\end{align*}
\begin{align*}
    J_k^q L_i^q (L_j^q)^* (\xi_{i_1} \ldots \xi_{i_M}) & = \sum_{l = 1}^{k} q^{l - 1} \delta_{j, i_l} e_i \xi_{i_1} \ldots \widehat{\xi_{i_l}} \ldots \xi_{i_k} \xi_{i_{M-k}} \ldots \xi_{i_{k+1}} \xi_{i_{M-k+1}} \ldots \xi_{i_M} + \\ & +  \sum_{l = k+1}^{M-k} q^{l - 1} \delta_{j, i_{l}} e_i \xi_{i_1} \ldots \xi_{i_{k-1}}  \xi_{i_{M-k}} \ldots \widehat{\xi_{i_{l}}} \ldots \xi_{i_{k}} \xi_{i_{M-k+1}} \ldots \xi_{i_M} + \\ & + \sum_{l = M-k+1}^{M} q^{l - 1} \delta_{j, i_l} e_i \xi_{i_1} \ldots \xi_{i_{k-1}} \xi_{i_{M-k-1}} \ldots \xi_{i_{k}} \xi_{i_{M-k}} \ldots \widehat{\xi_{i_l}} \ldots \xi_{i_M}.
\end{align*}
Thus
\[ L_i^q (L_j^q)^* J_k^q =  A_{i,j,k} + q^k L_i^q J_k^q (1^{\otimes k} \otimes (R_j^q)^* \otimes 1^{\otimes k}) + K_1, \ K_1 \in \mathbb{K}(\mathcal{F}^q), \]
\[ J_k^q L_i^q (L_j^q)^* = A_{i,j,k} + q^k J_k^q L_i^q (1^{\otimes k} \otimes (L_j^q)^* \otimes 1^{\otimes k}) + K_2, \ K_2 \in \mathbb{K}(\mathcal{F}^q),\]
\[ \lVert [j_k^q, \pi(L_i^q)\pi(L_j^q)^*] \rVert \leq q^k C_{J^q}\lVert L_i^q \rVert(C_{(R_j^q)^*} + C_{(L_j^q)^*}) \to 0. \]
\end{proof}

\begin{lemma}
For every $x \in (\oni^{L,R}_q)^{\mathbb{T}}$ one has $\lim_{k \to \infty} [|u_k|, x] = 0$.
\end{lemma}
\begin{proof}
We use the following inequality from \cite{nk}:
\[ \lVert [A^{\frac{1}{2}}, B] \rVert \leq \frac{5}{4} \lVert B \rVert \lVert [A, B] \rVert. \]
Notice that $u_k^* u_k = (j_k^q)^* j_k^q$.
\begin{align*}
     & \lim_{k \to \infty} \lVert [(u_k^* u_k)^\frac{1}{2}, x] \rVert \leq \frac{5}{4} \lVert x \rVert \lim_{k \to \infty} \lVert [(j_k^q)^* j_k, x]\rVert = \\ & = \frac{5}{4} \lVert x\rVert \lim_{k \to \infty} \lVert (j_k^q)^* [j_k^q, x] + [(j_k^q)^*, x] j_k^q \rVert \leq \\ & \leq \frac{5}{2} \lVert x \rVert \lVert j_k^q \rVert \lim_{k \to \infty} \lVert [j_k^q, x] \rVert \leq \\ & \leq \frac{5}{2} C_{J^q} \lVert x \rVert \lVert j^q \rVert \lim_{k \to \infty} \lVert [j_k^q, x] \rVert = 0.
\end{align*}
\end{proof}

\begin{lemma}\label{approxpror}
For every $x \in (\oni^{L, R}_q)^{\mathbb{T}}$ one has
\[ \lim_{k \to \infty} v_k^* x v_k = j^q x j^q. \]
\end{lemma}
\begin{proof}
\begin{align*}
    \lim_{k \to \infty} v_k^* x v_k & =   \lim_{k \to \infty} |u_k|^{-1} u_k^* x u_k |u_k|^{-1} =  \\ & = \lim_{k \to \infty} |u_k|^{-1} (j_k^q)^* j^q x j^q j_k^q |u_k|^{-1}  =  \\ & = \lim_{k \to \infty} |u_k|^{-1} j^q x j^q (j_k^q)^* j_k^q |u_k|^{-1}  + \\ & + \lim_{k \to \infty} |u_k|^{-1} [(j_k^q)^*, j^q x j^q] j_k^q |u_k|^{-1} = \\ & = j^q x j^q  + \\ & + \lim_{k \to \infty} [|u_k|^{-1}, j^q x j^q]|u_k| + \\ & + \lim_{k \to \infty} |u_k|^{-1} [(j_k^q)^*, j^q x j^q] j_k^q |u_k|^{-1}.
\end{align*}
Notice that for every $k$, $\lVert |u_k| \rVert = \lVert |u_k|^{-1} \rVert = \lVert j_k^q \rVert = \lVert u_k \rVert < C$. Thus
    \[ \lim_{k \to \infty} |u_k|^{-1} [(j_k^q)^*, j^q x j^q] j_k^q |u_k|^{-1} = 0, \]
    \[ \lim_{k \to \infty} [|u_k|^{-1}, j^q x j^q]|u_k| = 0. \]
\end{proof}

\begin{theorem}
$\oni_q^\mathbb{T}$ has flip approximation property.
\end{theorem}
\begin{proof}
The quotient mapping $id_{\otimes_{\alpha \rightarrow \min}} : (\oni_q^\mathbb{T} \otimes_\alpha \oni_q^\mathbb{T}, \sigma) \rightarrow (\oni_q^\mathbb{T} \otimes_{\min} \oni_q^\mathbb{T}, \sigma)$ is a $\mathbb{Z} / 2\mathbb{Z}$-equivariant contraction. Thus the sequence $id_{\otimes_{\alpha \rightarrow \min}}(v_k)$ implements flip approximation property for $\oni_q^\mathbb{T}$ by Lemma \ref{approxpror} and \ref{TensorIso}.
\end{proof}

\begin{corollary}
$\oni_q^\mathbb{T}$ is simple and nuclear.
\end{corollary}

\section{Asymptotic imbedding in $\mathfrak{U}_{n^\infty}$}
\begin{theorem}[\cite{RosenbergEffros}, Lemma 4.1]
Suppose that $A$ is a quasidiagonal unital $C^*$-algebra. Then $A$ has an asymptotic imbedding in $\mathfrak{U}_{n^\infty}$.
\end{theorem}
\begin{definition}
A $C^*$-algebra $A$ is called RFD if there exists a sequence $\pi_n : A \rightarrow F_n$, where $F_n$ is a finite-dimensional $C^*$-algebra such that 
\[ \lVert a \rVert = \sup_n \lVert \pi_n(a) \rVert. \]
\end{definition}
\begin{theorem}[\cite{Brown}, Example 3.15]
Suppose that $A$ is an RFD unital $C^*$-algebra. Then $A$ is quasidiagonal.
\end{theorem}
\begin{theorem}$\ghani_q^\mathbb{T}$ is RFD.
\end{theorem}
\begin{proof}
Notice that elements of $T \in \ghani_q^\mathbb{T}$ are precisely those for which $T(\mathcal{F}_m^q) \subset \mathcal{F}_m^q$. Consider $A_k = \{ T \in \ghani_q^\mathbb{T} : T \vert_{\mathcal{F}^q_m} = 0, \ m < k \}$. $A_k$ is an ideal in $\ghani_q^\mathbb{T}$ such that $\ghani_q^\mathbb{T} / A_k$ is a finite-dimensional algebra. Denote $\pi_k$ to be the quotient map. Suppose there exists $x \in \ghani_q^\mathbb{T}$ such that $x \in \ker \pi_k = A_k$ for every $k$. Then $ x\in \bigcap_{k = 0}^\infty A_k = \{0 \}$.
\end{proof}
\begin{theorem}[\cite{Brown}, Proposition 8.3]\label{quotientQD}
Assume $A$ is unital, nuclear and quasidiagonal, and $I$ is an ideal in $A$ which has an approximate unit consisting of projections which are quasicentral in $A$. Then $A / I$ is also quasidiagonal.
\end{theorem}
\begin{theorem}\label{quasi}
$\oni_q^\mathbb{T}$ is quasidiagonal.
\end{theorem}
\begin{proof}
Since the quotient mapping $\ghani_q \rightarrow \oni_q$ is $\mathbb{T}$-equivariant, \[ \oni_q^\mathbb{T} \simeq (\ghani_q / \mathbb{K}(\mathcal{F}^q))^\mathbb{T} \simeq \ghani_q^\mathbb{T} / \mathbb{K}(\mathcal{F}^q). \] Projections on the first $k$ components of $\mathcal{F}^q$ is an approximate unit of $\mathbb{K}(\mathcal{F}^q)$ which is central in $\ghani_q^\mathbb{T}$. Since $\ghani_q^\mathbb{T}$ is quasidiagonal, $\oni_q^\mathbb{T}$ is quasidiagonal by Theorem \ref{quotientQD}.
\end{proof}

\section{$\oni_q^\mathbb{T} \simeq \oni_0^\mathbb{T}$}

\begin{proposition}[\cite{Wegge}, Lemma L.1.4]
Let $\{A_i, \alpha_{ij} \}$ be a directed family (of $C^*$-algebras) where all the objects are isomorphic to an object $A$ with isomorphisms making the following diagram commutative
\[
\begin{tikzcd}
A_j \arrow[rd, "\simeq"] \arrow[rr, "\alpha_{ij}"] & & A_i \\
 & A \arrow[ur, "\simeq"] &
\end{tikzcd}
\]
Then
\[ \varinjlim (A_i, \alpha_{ij}) \simeq A. \]
\end{proposition}

\begin{proposition}
Let $\{A_n, \alpha_{nm} \}$ and $\{B_n, \beta_{nm}\}$ be directed families (of $C^*$-algebras). Assume there exists a sequence of isomorphisms $\lambda_n : A_n \rightarrow B_n$ making the following diagram commutative
\[
\begin{tikzcd}
A_{1} \arrow[r, "\alpha_1"] \arrow[d,"\lambda_1"] & A_2 \arrow[d, "\lambda_2"] \arrow[r, "\alpha_2"] & A_{3} \arrow[d, "\lambda_3"] \arrow[r, "\alpha_3"]  & \ldots \\
B_{1} \arrow[r, "\beta_1"] & B_{2} \arrow[r, "\beta_2"] & B_3 \arrow[r, "\beta_3"] & \ldots
\end{tikzcd}
\]
Then 
\[ \varinjlim (A_n, \alpha_{nm}) \simeq \varinjlim (B_n, \beta_{nm}). \]
\end{proposition}

\begin{lemma}
$\oni_q^\mathbb{T} \simeq M_n \otimes \oni_q^\mathbb{T}$.
\end{lemma}
\begin{proof}
Let $s_1, \ldots, s_n$ be generators of $\mathcal{O}_n \subset^\mathbb{T} \oni_q$. Define $\lambda : \oni_q^\mathbb{T} \rightarrow M_n(\oni_q^\mathbb{T})$
\[ \lambda(a) = \left( \begin{array}{ccc}
    s_1^* a s_1 & \ldots & s_1^* a s_n \\
     & \ldots & \\
    s_n^* a s_1 & \ldots & s_n^* a s_n 
\end{array} \right) \]
and $\lambda^{-1}$
\[ \lambda^{-1}\left( \begin{array}{ccc}
    a_{11} & \ldots & a_{1n} \\
     & \ldots & \\
    a_{n1} & \ldots & a_{nn} 
\end{array} \right) = \sum s_i a_{ij} s_j^*. \]
Since the inclusion of $\mathcal{O}_n$ is $\mathbb{T}$-equivariant, the maps are well-defined.
\end{proof}

Let $\varphi : \oni_q^\mathbb{T} \rightarrow \oni_q^\mathbb{T}$ given by $\varphi(a) = \sum_{i = 1}^n s_i a s_i^*$. Denote $\overline{\oni}_q^\mathbb{T}$ to be the following direct limit:
\[ \oni_q^\mathbb{T} \xrightarrow{\varphi} \oni_q^\mathbb{T} \xrightarrow{\varphi} \oni_q^\mathbb{T} \xrightarrow{\varphi} \ldots \rightarrow \overline{\oni}_q^\mathbb{T}. \]

\begin{theorem}\label{LimIso}
$\overline{\oni}_q^\mathbb{T} \simeq \mathfrak{U}_{n^\infty} \otimes \oni_q^\mathbb{T}$.
\end{theorem}
\begin{proof}
Define $\lambda_0 : \oni_q^\mathbb{T} \rightarrow \oni_q^\mathbb{T}$ to be $\id$ and $\lambda_{n + 1} : \oni_q^\mathbb{T} \rightarrow M_{n^k}(\oni_q^\mathbb{T})$ to be $(\id \otimes \lambda_n) \circ \lambda$. Notice that $\lambda_{n+1} \circ \varphi = (1 \otimes \id) \circ \lambda_n$. The following diagram with vertical isomorphisms commute
\[
\begin{tikzcd}
\oni_q^\mathbb{T} \arrow[r, "\varphi"] \arrow[d,"\lambda_0"] & \oni_q^\mathbb{T} \arrow[r, "\varphi"] \arrow[d, "\lambda_1"]  & \oni_q^\mathbb{T} \arrow[r, "\varphi"] \arrow[d, "\lambda_2"] & \ldots \\
\oni_q^\mathbb{T} \arrow[r, "1 \otimes \id"] & M_n(\oni_q^\mathbb{T})  \arrow[r, "1 \otimes \id"] & M_{n^2}(\oni_q^\mathbb{T}) \arrow[r,"1 \otimes \id"] & \ldots
\end{tikzcd}
\]
Top limit is $\overline{\oni}_q^\mathbb{T}$ and bottom limit is $\mathfrak{U}_{n^\infty} \otimes \oni_q^\mathbb{T}$.
\end{proof}

\begin{theorem}[\cite{RosenbergEffros}, Lemma 4.4]\label{LimQuasi}
Suppose $A$ is a direct limit of quasidiagonal $C^*$-algebras. Then $A$ is quasidiagonal.
\end{theorem}

\begin{theorem}[\cite{RosenbergEffros}, Corollary 2.4]\label{TensorFlip}
Suppose $A$ and $B$ have an approximately inner flip. Then $A \otimes B$ has an approximately inner flip.
\end{theorem}

\begin{theorem}
$\overline{\oni}_q^\mathbb{T} \simeq \mathfrak{U}_{n^\infty}$
\end{theorem}
\begin{proof}
Theorem \ref{LimQuasi} and \ref{quasi} imply that $\overline{\oni}_q^\mathbb{T}$ has an asymptotic embedding in $\mathfrak{U}_{n^\infty}$. Since $\oni_q^\mathbb{T}$ and $\mathfrak{U}_{n^\infty} \simeq \oni_0^\mathbb{T}$ have approximately inner flip, \ref{LimIso} and \ref{TensorFlip} imply that $\overline{\oni}_q^\mathbb{T}$ has an approximately inner flip. Also \ref{LimIso} implies that
\[\mathfrak{U}_{n^\infty} \otimes \overline{\oni}_q^\mathbb{T} \simeq  (\mathfrak{U}_{n^\infty} \otimes \mathfrak{U}_{n^\infty}) \otimes \oni_q^\mathbb{T} \simeq \mathfrak{U}_{n^\infty} \otimes \oni_q^\mathbb{T} \simeq \overline{\oni}_q^\mathbb{T}. \]
Thus by Theorem \ref{RosenbergUHF},
$\overline{\oni}_q^\mathbb{T} \simeq \mathfrak{U}_{n^\infty}$.
\end{proof}

\begin{definition}
Let $B, A \subset \mathbb{B}(\mathcal{H})$ be $C^*$-algebras. We say that $A \subset_\gamma B$ if for every $a \in A$ such that $\lVert a \rVert \leq 1$ there is $b \in B$ such that $\lVert b - a \rVert < \gamma$.
\end{definition}

\begin{theorem}[\cite{Christensen}, Theorem 6.10]\label{approximation}
Let $A \subset_\gamma B$, let $\eta = 2(n+1)(2\gamma + \gamma^2)(2 + 2\gamma + \gamma^2)$, and suppose that $A$ is separable with nuclear
dimension at most $n$. If $\eta < 1/210000$, then $A$ embeds into $B$. Moreover, for each finite subset $X$ of the unit ball of $A$, there exists an embedding $\theta : A\rightarrow B$ with $\lVert \theta(x) - x \rVert \leq 20 \gamma^{\frac{1}{2}}, \ x \in X$.
\end{theorem}

\begin{theorem}
$\oni_q^\mathbb{T}$ is an AF-algebra.
\end{theorem}
\begin{proof}
We use the following characterization of AF-algebras: $A$ is an AF-algebra if for every subset $\{x_1, \ldots x_n\} \subset A$ and $\epsilon > 0$ there exists a finite-dimensional subalgebra $B$ and a subset $\{b_1, \ldots, b_n \} \subset B$ such that $\lVert x_i - b_i \rVert < \epsilon$. 

Let $i_k : \oni_q^\mathbb{T} \rightarrow \overline{\oni}_q^\mathbb{T}$ be monomorphisms induced from the direct system. Let $\{x_1, \ldots, x_n \} \subset \oni_q^\mathbb{T}$ and $\epsilon > 0$. Since $\overline{\oni}_q^\mathbb{T}$ is an AF-algebra, there exists a finite-dimensional subalgebra $B$ and a subset $\{b_1, \ldots, b_n\} \subset B$ such that $\lVert i_0(x_i) - b_i \rVert < \eta$ (we can choose $B$ to be $C^*(b_1, \ldots, b_n)$). Since $\overline{\oni}_q^\mathbb{T}$ is a direct limit, there exists $\{a_1, \ldots, a_n \} \subset A$ and $k$ such that $\lVert b_i - i_k(a_i) \rVert < \theta$. 
Let $\overline{\varphi} : \overline{\oni}_q^\mathbb{T} \rightarrow \overline{\oni}_q^\mathbb{T}$ be an automorphism induced by the following diagram:
\[
\begin{tikzcd}
\oni_q^\mathbb{T} \arrow[r, "\varphi"] \arrow[d,"\varphi"] & \oni_q^\mathbb{T} \arrow[r, "\varphi"] \arrow[d, "\varphi"]  & \oni_q^\mathbb{T} \arrow[r, "\varphi"] \arrow[d, "\varphi"] & \ldots \arrow[r] & \overline{\oni}_q^\mathbb{T} \arrow[d, "\overline{\varphi}"] \\
\oni_q^\mathbb{T} \arrow[r, "\varphi"] & \oni_q^\mathbb{T}  \arrow[r, "\varphi"] & \oni_q^\mathbb{T} \arrow[r,"\varphi"] & \ldots \arrow[r] & \overline{\oni}_q^\mathbb{T}
\end{tikzcd}
\]
We prove that if we choose $\theta$ small enough, $\overline{\varphi}^{-k}(B) \subset_\eta i_0(\oni_q^\mathbb{T})$. Notice that it is equivalent to $B \subset_\eta i_k(\oni_q^\mathbb{T})$.

Let $N = \dim B$. Since $B$ is finite-dimensional, there exists a set of multiindices $\{ \mu_1, \ldots, \mu_N \}$, where $\mu_i = (\alpha_1^i, \ldots, \alpha_{n_i}^i)$ and $\alpha_j \in \{-n, \ldots, -1, 1, \ldots, n \}$ such that $\{ b_{\mu_1}, \ldots, b_{\mu_N} \}$ is a basis of $B$. On $B$ we have Frobenius inner product and Frobenius norm. By change of basis of $B$ via some linear operator $T : B \rightarrow B$ we can choose basis $c_{\mu_i} = T(b_{\mu_i})$ which is orthonormal with respect to the Frobenius inner product. Denote $a_{\mu_i} = i_k(a_{\alpha_1^i})\ldots i_k(a_{\alpha_{n_i}^i})$. Then there exists a sequence of polynomials $P_1, P_2, \ldots$ with positive coefficients and $P_i(0) = 0$ such that $\lVert b_{\mu} - a_{\mu} \rVert \leq P_{|\mu|}(\theta)$ for every multiindex $\mu$. One can prove this by induction: $\lVert b_i - i_k(a_i) \rVert < \theta$, so $P_1(x) = x$. Let $K = \max \{ \lVert b_1 \rVert, \ldots, \lVert b_n \rVert \}$ and notice that $\lVert a_i \rVert \leq \lVert b_i \rVert + \theta$
\begin{align*}
    \lVert b_\mu - a_\mu \rVert & = \lVert (b_{\alpha_1} - a_{\alpha_1}) b_{\mu \setminus 1} + a_{\alpha_1} b_{\mu \setminus 1} - a_{\alpha_1} a_{\mu \setminus 1} \rVert \leq \\ & \leq \theta K^{|\mu| - 1} + (K + \theta)P_{|\mu| - 1}(\theta).
\end{align*}
Thus we can define $P_{n + 1}(x) = x K^n + (K + x)P_n(x)$. By induction one can prove that $P_n(x) = (x + K)^n - K^n$.

Suppose we choose $b \in B$ with $\lVert b \rVert \leq 1$. Then $b = \sum_{i = 1}^N C_i^{(b)} c_{\mu_i}$. Since all norms on $B$ are equivalent, there exists a constant $K$ such that \[ \lVert b \rVert \geq L \lVert b \rVert_F. \] Frobenius norm has property $\lVert a + b \rVert_F = \lVert a \rVert_F + \lVert b \rVert_F + 2 \langle a, b \rangle_F$. Since the basis was chosen orthonormal,
\[ 1 \geq \lVert b \rVert \geq L \lVert \sum C_i^{(b)} c_{\mu_i} \rVert_F = L \sum_{i = 1}^N |C_i^{(b)}|^2.\] Thus $|C_i^{(b)}| \leq \frac{1}{\sqrt{L}}$. Let $\alpha = \max_{i = 1}^N |\mu_i|$. Suppose that $\theta$ was chosen to be less than $1$. Then
\begin{align*}
     \lVert \sum_{i = 1}^N C_i^{(b)} (c_{\mu_i} - a_{\mu_i}) \rVert & \leq  \frac{1}{\sqrt{L}} \sum_{i = 1}^N (\lVert T(b_{\mu_i}) - b_{\mu_i} \rVert + \lVert b_{\mu_i} - a_{\mu_i} \rVert) \leq \\ & \leq \frac{1}{\sqrt{L}} \sum_{i = 1}^N ((\lVert T \rVert + 1) \lVert b_{\mu_i} \rVert + P_{|\mu_i|}(\theta)) \leq \\ & \leq \frac{\lVert T \rVert + 1}{\sqrt{L}} \sum_{i = 1}^N (\theta^{|\mu_i|} + P_{|\mu_i|}(\theta)) \leq \\ & \leq \frac{(\lVert T \rVert + 1)N}{\sqrt{L}} (\theta^{\alpha} + (\theta + K)^\alpha - K^\alpha).
\end{align*}
The last inequality follows from that $\theta^\alpha + (\theta + K)^n - K^n$ increases on the interval $(0,1)$ when $\alpha$ increases. Since $(\theta^{\alpha} + (\theta + K)^\alpha - K^\alpha)(0) = 0$, we can find $\theta$ sufficiently small for \[ \frac{(\lVert T \rVert + 1)N}{\sqrt{L}} (\theta^{\alpha} + (\theta + K)^\alpha - K^\alpha) < \eta .\]

Since $B$ has finite nuclear dimension, by Theorem \ref{approximation} if we take $\eta$ small enough, there exists an embedding $f_\eta : \overline{\varphi}^{-k}(B) \rightarrow i_0(\oni_q^\mathbb{T})$ such that $\lVert f_\eta(\overline{\varphi}^{-k}(b_i)) - \overline{\varphi}^{-k}(b_i) \rVert \leq 20\eta^{\frac{1}{2}}$. Let $c_i^\eta \in \oni_q^\mathbb{T}$ be such that $i_0(c_i^\eta) = f_\eta(\overline{\varphi}^{-k}(b_i))$.
\begin{align*}
    \lVert x_i - \varphi^k(c_i^\eta) \rVert & = \lVert i_0(x_i) - i_0(\varphi^k(c_i^\eta)) \rVert = \\ & =\lVert i_0(x_i) - \overline{\varphi}^{k}(i_0(c_i^\eta)) \rVert \leq \\
    & \leq \lVert i_0(x_i) - b_i \rVert + \lVert b_i - \overline{\varphi}^{k}(i_0(c_i^\eta)) \rVert \leq \\ & \leq \eta + \lVert \overline{\varphi}^{-k}(b_i) - i_0(c_i^\eta) \rVert = \\ & = \eta + \lVert \overline{\varphi}^{-k}(b_i) - f_\eta(\overline{\varphi}^{-k}(b_i)) \rVert \leq \eta + 20 \eta^{\frac{1}{2}}. 
\end{align*}
Take $\eta$ such that $\epsilon = \eta + 20\eta^{\frac{1}{2}}$. Since $C^*(\varphi^k(c_1^\eta), \ldots, \varphi^k(c_n^\eta))$ is finite-dimensional, $\oni_q^\mathbb{T}$ satisfies the characterization of AF-algebras.

\end{proof}

\begin{theorem}\label{UHF}
$\oni_q^\mathbb{T} \simeq \mathfrak{U}_{n^\infty}$.
\end{theorem}
\begin{proof}
\
\begin{enumerate}
    \item Kunneth theorem: $K_0(\oni_q^\mathbb{T} \otimes \mathfrak{U}_{n^\infty}) \simeq K_0(\oni_q^\mathbb{T}) \otimes_\mathbb{Z} K_0(\mathfrak{U}_{n^\infty})$. Thus
\[ K_0(\oni_q^\mathbb{T}) \otimes_\mathbb{Z} \mathbb{Z}[\frac{1}{n}] \simeq \mathbb{Z}[\frac{1}{n}]. \]
    \item Rank of an abelian group $A$ coincides with $\dim_\mathbb{Q} (A \otimes_\mathbb{Z} \mathbb{Q})$. Thus rank of $K_0(\oni_q^\mathbb{T})$ is
\begin{align*}
    \dim(K_0(\oni_q^\mathbb{T}) \otimes_\mathbb{Z} \mathbb{Q}) & = \dim(K_0(\oni_q^\mathbb{T}) \otimes (\mathbb{Z}[\frac{1}{n}] \otimes_\mathbb{Z} \mathbb{Q})) = \\ & = \dim((K_0(\oni_q^\mathbb{T}) \otimes \mathbb{Z}[\frac{1}{n}]) \otimes_\mathbb{Z} \mathbb{Q}) = \\ & = \dim(\mathbb{Z}[\frac{1}{n}] \otimes_\mathbb{Z} \mathbb{Q}) = \\ & = \dim(\mathbb{Q}) = 1.
\end{align*} 
    \item It is also clear that $K_0(\oni_q^\mathbb{T})$ is torsion-free, since it is a subgroup ($K_0(\oni_q^\mathbb{T}) \otimes_\mathbb{Z} 1 \hookrightarrow \mathbb{Z}[\frac{1}{n}]$) of a torsion-free group.
    \item In an abelian group of rank 1 for every two nontrivial elements $a, b$ there are numbers $n, m$ such that $na + mb = 0$. 
    \item Suppose one has an endomorphism of a torsion-free abelian group of rank 1. Then it is determined by it's value on an arbitrary element. Let $\varphi(a) = x$. Then for any $b$ in group $0 = \varphi(na + mb) = nx + m\varphi(b)$. Since the group is torsion-free, $\varphi(b)$ is determined uniquely from this equation.
    \item Recall $\varphi : \oni_q^\mathbb{T} \rightarrow \oni_q^\mathbb{T}$. Since $\varphi(1) = 1$, $K_*(\varphi) = id$. Denote $K_*^{\leq}$ to be the functor of ordered group, which is an invariant of AF-algebras. It commutes with direct limits. Take $K_*^{\leq}$ of the sequence
\[ \oni_q^\mathbb{T} \xrightarrow{\varphi} \oni_q^\mathbb{T} \xrightarrow{\varphi} \oni_q^\mathbb{T} \xrightarrow{\varphi} \ldots \rightarrow \overline{\oni}_q^\mathbb{T}. \]
Since $K_*(\varphi) = id$, it follows that $K_*^{\leq}(\oni_q^\mathbb{T}) \simeq K_*^{\leq}(\overline{\oni}_q^\mathbb{T}) \simeq K_*^{\leq}(\mathfrak{U}_{n^\infty})$. Assume $\varphi \in \Aut(\mathbb{Z}[\frac{1}{n}])$. Then $\varphi(x) = \frac{1}{n^b}$. Thus the class of $1 \in K_0(\oni_q^\mathbb{T})$ is equal to $\frac{1}{n^b}$. By the classification theorem for AF-algebras, $\oni_q^\mathbb{T} \otimes M_{n^{-b}} \simeq \mathfrak{U}_{n^\infty}$ if $b < 0$ and $\oni_q^\mathbb{T} \simeq \mathfrak{U}_{n^\infty} \otimes M_{n^{b}}$ if $b > 0$. By Lemma 4.18, $\oni_q^\mathbb{T} \simeq \mathfrak{U}_{n^\infty}$.
\end{enumerate}
\end{proof}

\section{$\oni_q \simeq \oni_0$}

\begin{theorem}\label{Paschke}
$\oni_q = C^*((\oni_q^L)^\mathbb{T}, s_1)$.
\end{theorem}
\begin{proof}
Since $s_1 = (\rho_L^{+})^{\frac{1}{2}} \pi(L_1^q)$, $a_1 = \pi(L_1^q) \in C^*((\oni_q^L)^\mathbb{T}, s_1)$. We prove that $a_k = \pi(L_k^q) \in C^*((\oni_q^L)^\mathbb{T}, s_1)$:
\begin{align*}
    & (a_k\sum_{l = 0}^\infty (-q)^l (a_1 a_1^*)^l a_1^*) a_1  = 
    a_k\sum_{l = 0}^\infty (-q)^l (a_1 a_1^*)^l + q(-q)^l (a_1 a_1^*)^{l+1} = a_k.
\end{align*}
\end{proof}

\begin{theorem}\label{crossedproduct}
\[ \oni_q \simeq \oni_q^\mathbb{T} \rtimes_{\mathsf{Ad}(s_1)} \mathbb{N}. \]
\end{theorem}
\begin{proof}
One can easily see that $(\oni_q^\mathbb{T}, s_1)$ satisfy all conditions of Theorem \ref{OrtegaEquivalenceModels}. Combining with Theorem \ref{Paschke}, we get the result.
\end{proof}

\begin{theorem}
$\oni_q$ is simple, purely infinite, nuclear and satisfies UCT.
\end{theorem}
\begin{proof}
Simplicity is a corollary of Theorem \ref{crossedproduct}, \ref{UHF} and \ref{OrtegaEquivalenceModels}.

Pure infiniteness follows from Theorem \ref{crossedproduct}, \ref{UHF} and \ref{OrtegaPF}.

Nuclearity and UCT follows from $\oni_q^\mathbb{T}$ being AF.
\end{proof}

\begin{theorem}\label{KTheoryOni}
\[ K_*(\oni_q) \simeq K_*(\oni_0) \simeq \mathbb{Z} / (n - 1) \mathbb{Z} \oplus 0. \]
Moreover, $[1_{\oni_q}] = [1_{\oni_0}] = 1 \in \mathbb{Z} / (n - 1) \mathbb{Z}$.
\end{theorem}
\begin{proof}
We use Pimsner-Voiculescu sequence for Stacey crossed products:
\[
\begin{tikzcd}
K_0(\oni_q^\mathbb{T}) \arrow[rr, "1 - K_0(\mathsf{Ad}(s_1))"] & & K_0(\oni_q^\mathbb{T}) \arrow[rr, "\iota"] & & K_0(\oni_q) \arrow[d] \\
K_1(\oni_q) \arrow[u] & & K_1(\oni_q^\mathbb{T}) \arrow[ll, "\iota"] & & K_1(\oni_q^\mathbb{T}) \arrow[ll, "1 - K_1(\mathsf{Ad}(s_1))"]
\end{tikzcd}
\]
Since $K_0(\oni_q^\mathbb{T}) \simeq \mathbb{Z}[\frac{1}{n}]$, $K_0(\mathsf{Ad}(s_1))$ is determined by the image of $[1]$:
\[ K_0(\mathsf{Ad}(s_1))([1]) = [s_1 s_1^*] = \frac{1}{n} \sum_{i = 1}^n [s_i s_i^*] = \frac{1}{n} [\varphi(1)] = \frac{1}{n}[1]. \]
Thus the sequence can be rewritten:
\[
\begin{tikzcd}
\mathbb{Z}[\frac{1}{n}] \arrow[r, "\frac{n-1}{n}"] & \mathbb{Z}[\frac{1}{n}] \arrow[r, "\iota"] & K_0(\oni_q) \arrow[d] \\
K_1(\oni_q) \arrow[u] & 0 \arrow[l] & 0 \arrow[l]
\end{tikzcd}
\]
Hence
\[K_0(\oni_q) \simeq \mathbb{Z}[\frac{1}{n}] / \ker \iota \simeq \mathbb{Z}[\frac{1}{n}] / (n-1) \mathbb{Z}[\frac{1}{n}] \simeq \mathbb{Z} / (n-1) \mathbb{Z}. \]
\[K_1(\oni_q) \simeq \ker \frac{n-1}{n} = 0. \]
Since $K_0(\iota)(\frac{a}{n^b}) = a \mod \mathbb{Z} / (n-1)\mathbb{Z}$ and $[1_{\oni_q^\mathbb{T}}] = n^b \in \mathbb{Z}[\frac{1}{n}]$, \[ [1_{\oni_q}] = K_0(\iota)([1_{\oni_q^\mathbb{T}}]) = K_0(\iota)(n^b) = 1 \in \mathbb{Z} / (n-1)\mathbb{Z}. \]
\end{proof}

\begin{corollary}
\[ \oni_q \simeq \oni_0. \]
\end{corollary}

\section{$\ghani_q \simeq \ghani_0$}
Consider the following extension
\[ \mathcal{E}_q : 0 \rightarrow \mathbb{K} \rightarrow \ghani_q \rightarrow \oni_q \simeq \mathcal{O}_n \rightarrow 0. \]
\begin{theorem}
\[
K_{six}^u(\mathcal{E}_q) = 
\begin{tikzcd}
\mathbb{Z} \arrow[rr, "\cdot (n-1)"] & & \mathbb{Z} \arrow[rr, "\mod (n-1)"] & & \mathbb{Z}_{n-1} \arrow[d] \\
0 \arrow[u] & & 0 \arrow[ll] & & 0 \arrow[ll]
\end{tikzcd}
\]
with $[1_{\ghani_q}] = 1 \in \mathbb{Z}$ and $[1_{\oni_q}] = 1 \in \mathbb{Z}_{n-1}$.
\end{theorem}
\begin{proof}
Apply six-term exact sequence to the extension $\mathcal{E}_q$:
\[
K_{six}(\mathcal{E}_q) = \mathcal{E}_q = 
\begin{tikzcd}
\mathbb{Z} \arrow[rr] & & K_0(\ghani_q) \arrow[rr] & & \mathbb{Z}_{n-1} \arrow[d] \\
0 \arrow[u] & & K_1(\ghani_q) \arrow[ll] & & 0 \arrow[ll]
\end{tikzcd}
\]
There are two types of group extensions of $\mathbb{Z}_{n-1}$ by $\mathbb{Z}$ : they are either $\mathbb{Z}$ or $\mathbb{Z} \oplus \mathbb{Z}_{n-1}$. If the second is the case then $\mathcal{E}_q$ is a trivial extension. Since it is also essential, by Voiculescu theorem that would mean that $\ghani_q \simeq \mathbb{K} \oplus \oni_q$, which is not the case, because $P_\Omega = 1_{\ghani_q} - \sum S_i S_i^*$. Thus
\[ K_0(\ghani_q) \simeq \mathbb{Z}, \ K_1(\ghani_q) \simeq 0. \]

In Theorem \ref{KTheoryOni} we have already shown that $[1_{\oni_q}] = 1$. Since all surjective maps $\mathbb{Z} \rightarrow \mathbb{Z}_{n-1}$ has form $\Lambda_b : a \mapsto ba \mod (n-1)$ for some $b$ coprime with $n-1$ and $1_{\ghani_q}$ maps to $[1_{\oni_q}]$, we have that
\[
K_{six}(\mathcal{E}_q) = 
\begin{tikzcd}
\mathbb{Z} \arrow[rr, "\cdot \pm(n-1)"] & & \mathbb{Z} \arrow[rr, "\Lambda_b"] & & \mathbb{Z}_{n-1} \arrow[d] \\
0 \arrow[u] & & 0 \arrow[ll] & & 0 \arrow[ll]
\end{tikzcd}
\]
and $[1_{\ghani_q}] = 1 + k(n-1)$ for some $k \in \mathbb{Z}$. Since $P_\Omega$ is minimal projection in $\mathbb{K} \subset \ghani_q$ and $K_0(\mathbb{K}) \hookrightarrow K_0(\ghani_q)$ is injective, $[P_\Omega]\mathbb{Z} = (n-1)\mathbb{Z}$. But then
\[ \pm (n-1) = [P_\Omega] = [\sum S_i S_i^*] = [\sum S_i^* S_i] = n[1_{\ghani_q}], \]
so either $n(1 + k(n-1)) = n-1$ or $n(1 + k(n-1)) = -n+1$ which holds only for $k = 0$. Thus $[1_{\ghani_q}] = 1$.
\end{proof}
Since considerations of the theorem above did not depend of $q$, by Theorem \ref{Gabe}
\begin{corollary}\label{mainiso}
\[ \ghani_q \simeq \ghani_0. \]
\end{corollary}

\medskip 

{\bf Acknowledgement. } Thank you Quantum Fox for sharing this journey with me.


\begin{thebibliography}{9}
\providecommand{\url}[1]{{#1}}
\providecommand{\urlprefix}{URL }
\expandafter\ifx\csname urlstyle\endcsname\relax
  \providecommand{\doi}[1]{DOI~\discretionary{}{}{}#1}\else
  \providecommand{\doi}{DOI~\discretionary{}{}{}\begingroup
  \urlstyle{rm}\Url}\fi



\bibitem{Black}
Blackadar, B.: {{$K$}-theory for operator algebras}, Second edition, \emph{{Mathematical
  Sciences Research Institute Publications}}, vol.~5.
\newblock Cambridge University Press, Cambridge (1998).
\newblock \doi{10.1007/978-1-4613-9572-0}.
\newblock \urlprefix\url{https://doi.org/10.1007/978-1-4613-9572-0}

\bibitem{BoLyt}
Bo\.{z}ejko, M., Lytvynov, E., Wysocza{\'n}ski, J.: {Fock representations of
  {$Q$}-deformed commutation relations}.
\newblock J. Math. Phys. \textbf{58}(7), 073501, 19 (2017).
\newblock \doi{10.1063/1.4991671}.
\newblock \urlprefix\url{https://doi.org/10.1063/1.4991671}

\bibitem{BoSpe2}
Bo\.{z}ejko, M., Speicher, R.: {An example of a generalized {B}rownian motion}.
\newblock Comm. Math. Phys. \textbf{137}(3), 519--531 (1991).
\newblock \urlprefix\url{http://projecteuclid.org/euclid.cmp/1104202738}

\bibitem{BoSpe}
Bo\.{z}ejko, M., Speicher, R.: {Completely positive maps on {C}oxeter groups,
  deformed commutation relations, and operator spaces}.
\newblock Math. Ann. \textbf{300}(1), 97--120 (1994).
\newblock \doi{10.1007/BF01450478}.
\newblock \urlprefix\url{https://doi.org/10.1007/BF01450478}

\bibitem{coburn}
Coburn, L.A.: {The {$C^{\ast} $}-algebra generated by an isometry}.
\newblock Bull. Amer. Math. Soc. \textbf{73}, 722--726 (1967).
\newblock \doi{10.1090/S0002-9904-1967-11845-7}.
\newblock \urlprefix\url{https://doi.org/10.1090/S0002-9904-1967-11845-7}

\bibitem{cun}
Cuntz, J.: {Simple {$C\sp*$}-algebras generated by isometries}.
\newblock Comm. Math. Phys. \textbf{57}(2), 173--185 (1977).
\newblock \urlprefix\url{http://projecteuclid.org/euclid.cmp/1103901288}

\bibitem{Dalet}
Daletskii, A., Kalyuzhny, A., Lytvynov, E., Proskurin, D.: Fock representations of multicomponent (particularly non-Abelian anyon) commutation relations. 
\newblock Reviews in Mathematical Physics, \textbf{32}(05),  2030004 (2020)

\bibitem{dav}
Davidson, K.R.: {{$C^*$}-algebras by example}, \emph{{Fields Institute
  Monographs}}, vol.~6.
\newblock American Mathematical Society, Providence, RI (1996).
\newblock \doi{10.1090/fim/006}.
\newblock \urlprefix\url{https://doi.org/10.1090/fim/006}

\bibitem{dn}
Dykema, K., Nica, A.: {On the {F}ock representation of the {$q$}-commutation
  relations}.
\newblock J. Reine Angew. Math. \textbf{440}, 201--212 (1993)


\bibitem{fiv}
Fivel, D.I.: {Errata: ``{I}nterpolation between {F}ermi and {B}ose statistics
  using generalized commutators'' [{P}hys. {R}ev. {L}ett. {\bf 65} (1990), no.
  27, 3361--3364; {MR}1084327 (91m:81108)]}.
\newblock Phys. Rev. Lett. \textbf{69}(13), 2020 (1992).
\newblock \doi{10.1103/PhysRevLett.69.2020}.
\newblock \urlprefix\url{https://doi.org/10.1103/PhysRevLett.69.2020}

\bibitem{gisselson}
Giselsson, O.: {The Universal $C^*$-Algebra of the Quantum Matrix Ball and its
  Irreducible $*$-Representations} (2018).
\newblock \urlprefix\url{http://arxiv.org/pdf/1801.10608v2}

\bibitem{green}
Greenberg, O.W.: {Particles with small violations of {F}ermi or {B}ose
  statistics}.
\newblock Phys. Rev. D (3) \textbf{43}(12), 4111--4120 (1991).
\newblock \doi{10.1103/PhysRevD.43.4111}.
\newblock \urlprefix\url{https://doi.org/10.1103/PhysRevD.43.4111}

\bibitem{jeu_pinto}
de~Jeu, M., Pinto, P.R.: {The structure of doubly non-commuting isometries}
\newblock Adv. Math. \textbf{368} (2020), 107149, 35 pp.
\newblock \urlprefix\url{https://arxiv.org/abs/1801.09716}

\bibitem{jps}
{J{\o}rgensen}, P.E.T., Proskurin, D.P., Samo\u{\i}lenko, Y.S.: {The kernel of
  {F}ock representations of {W}ick algebras with braided operator of
  coefficients}.
\newblock Pacific J. Math. \textbf{198}(1), 109--122 (2001).
\newblock \doi{10.2140/pjm.2001.198.109}.
\newblock \urlprefix\url{https://doi.org/10.2140/pjm.2001.198.109}

\bibitem{jps2}
{J{\o}rgensen}, P.E.T., Proskurin, D.P., Samo\u{\i}lenko, Y.S.: {On
  {$C^*$}-algebras generated by pairs of {$q$}-commuting isometries}.
\newblock J. Phys. A \textbf{38}(12), 2669--2680 (2005).
\newblock \doi{10.1088/0305-4470/38/12/009}.
\newblock \urlprefix\url{https://doi.org/10.1088/0305-4470/38/12/009}

\bibitem{jsw2}
Jorgensen, P.E.T., Schmitt, L.M., Werner, R.F.: {{$q$}-canonical commutation
  relations and stability of the {C}untz algebra}.
\newblock Pacific J. Math. \textbf{165}(1), 131--151 (1994).
\newblock \urlprefix\url{http://projecteuclid.org/euclid.pjm/1102621916}

\bibitem{jsw}
Jorgensen, P.E.T., Schmitt, L.M., Werner, R.F.: {Positive representations of
  general commutation relations allowing {W}ick ordering}.
\newblock J. Funct. Anal. \textbf{134}(1), 33--99 (1995).
\newblock \doi{10.1006/jfan.1995.1139}.
\newblock \urlprefix\url{https://doi.org/10.1006/jfan.1995.1139}


\bibitem{nk}
Kennedy, M., Nica, A.: {Exactness of the {F}ock space representation of the
  {$q$}-commutation relations}.
\newblock Comm. Math. Phys. \textbf{308}(1), 115--132 (2011).
\newblock \doi{10.1007/s00220-011-1323-9}.
\newblock \urlprefix\url{https://doi.org/10.1007/s00220-011-1323-9}

\bibitem{kab}
Kim, C.S., Proskurin, D.P., Iksanov, A.M., Kabluchko, Z.A.: {The generalized
  {CCR}: representations and enveloping {$C^\ast$}-algebra}.
\newblock Rev. Math. Phys. \textbf{15}(4), 313--338 (2003).
\newblock \doi{10.1142/S0129055X03001618}.
\newblock \urlprefix\url{https://doi.org/10.1142/S0129055X03001618}

\bibitem{Kirchberg}
Kirchberg, E.: {Exact $C^*$-algebras, tensor products, and the classification of purely infinite algebras}.
\newblock Proceedings of the International Congress of Mathematicians,  Vol. 1, 2 (Z\"urich 1994), 954--954, Birkh\"auser, Basel, 1995.

\bibitem{Klimek}
Klimek, S., Lesniewski, A.: {A two-parameter quantum deformation of the unit
  disc}.
\newblock J. Funct. Anal. \textbf{115}(1), 1--23 (1993).
\newblock \doi{10.1006/jfan.1993.1078}.
\newblock \urlprefix\url{https://doi.org/10.1006/jfan.1993.1078}

\bibitem{Klimyk}
Klimyk, A., Schm{\"u}dgen, K.: {Quantum groups and their representations}.
\newblock {Texts and Monographs in Physics}. Springer-Verlag, Berlin (1997).
\newblock \doi{10.1007/978-3-642-60896-4}.
\newblock \urlprefix\url{https://doi.org/10.1007/978-3-642-60896-4}

\bibitem{LiM}
Liguori, A., Mintchev, M.: {Fock representations of quantum fields with
  generalized statistics}.
\newblock Comm. Math. Phys. \textbf{169}(3), 635--652 (1995).
\newblock \urlprefix\url{http://projecteuclid.org/euclid.cmp/1104272858}

\bibitem{mac}
Macfarlane, A.J.: {On {$q$}-analogues of the quantum harmonic oscillator and
  the quantum group {${\rm SU}(2)_q$}}.
\newblock J. Phys. A \textbf{22}(21), 4581--4588 (1989).
\newblock \urlprefix\url{http://stacks.iop.org/0305-4470/22/4581}

\bibitem{mar}
Marcinek, W.: {On commutation relations for quons}.
\newblock Rep. Math. Phys. \textbf{41}(2), 155--172 (1998).
\newblock \doi{10.1016/S0034-4877(98)80173-0}.
\newblock \urlprefix\url{https://doi.org/10.1016/S0034-4877(98)80173-0}

\bibitem{MPe}
Meljanac, S., Perica, A.: {Generalized quon statistics}.
\newblock Modern Phys. Lett. A \textbf{9}(35), 3293--3299 (1994).
\newblock \doi{10.1142/S0217732394003117}.
\newblock \urlprefix\url{https://doi.org/10.1142/S0217732394003117}

\bibitem{yakym}
Ostrovska, O.,  Yakymiv, R.:
{On isometries satisfying deformed commutation relations},
\newblock Methods Funct. Anal. Topology \textbf{25} (2019), no. 2, 152--160.

\bibitem{Philips}
Phillips, N.C.: {A classification theorem for nuclear purely infinite simple
  {$C^*$}-algebras}.
\newblock Doc. Math. \textbf{5}, 49--114 (2000)

\bibitem{Popescu} Popescu, G.: {Doubly $\Lambda$-commuting row isometries, universal models, and classification}
\newblock Journal  Funct. Anal. \textbf{279} (12), 108798, (2020)
Journal of Functional Analysis

\bibitem{prolett}
Proskurin, D.: {Stability of a special class of {$q_{ij}$}-{CCR} and extensions
  of higher-dimensional noncommutative tori}.
\newblock Lett. Math. Phys. \textbf{52}(2), 165--175 (2000).
\newblock \doi{10.1023/A:1007668304707}.
\newblock \urlprefix\url{https://doi.org/10.1023/A:1007668304707}

\bibitem{woronowicz}
Pusz, W., Woronowicz, S.L.: {Twisted second quantization}.
\newblock Rep. Math. Phys. \textbf{27}(2), 231--257 (1989).
\newblock \doi{10.1016/0034-4877(89)90006-2}.
\newblock \urlprefix\url{https://doi.org/10.1016/0034-4877(89)90006-2}

\bibitem{vaksman1}
Vaksman, L.L.: {Quantum bounded symmetric domains}, \emph{{Translations of
  Mathematical Monographs}}, vol. 238.
\newblock American Mathematical Society, Providence, RI (2010).
\newblock Translated from the Russian manuscript and with a foreword by Olga
  Bershtein and Sergey D. Sinelshchikov

\bibitem{Black}
Blackadar, B.: {{$K$}-theory for operator algebras}, \emph{{Mathematical
  Sciences Research Institute Publications}}, vol.~5.
\newblock Springer-Verlag, New York (1986).
\newblock \doi{10.1007/978-1-4613-9572-0}.

\bibitem{zag}
Zagier, D.: {Realizability of a model in infinite statistics}.
\newblock Comm. Math. Phys. \textbf{147}(1), 199--210 (1992).
\newblock \urlprefix\url{http://projecteuclid.org/euclid.cmp/1104250533}

\bibitem{RosenbergEffros}
Effros, E. G., Rosenberg J.: {C*-algebras with approximately inner flip}.
\newblock Pacific J. Math. \textbf{77}, 417-443 (1978).

\bibitem{Bozejko}
Bozejko M., Speicher R.: {An example of a generalized Brownian
motion}. \newblock Commun. Math. Phys. \textbf{137}, 519--531, (1991).

\bibitem{Ortega}
Ortega E., Pardo E.: {The structure of Stacey cross products by endomorphisms}. \newblock Operator algebra and dynamics. \textbf{58} 239-252, (2013).

\bibitem{Wegge}
Wegge-Olsen N. E.: {K-theory and C*-algebras}. \newblock OxfordUniversityPress (1993)

\bibitem{Brown}
Brown N.P.: {On quasidiagonal C*-algebras}. \newblock Proceedings of 1999 US-Japan conference on Operator Algebras (to appear)

\bibitem{Christensen}
Christensen E., Sinclair A.M., Smith R.R., White S.A., Winter W.: {Perturbations of nuclear C*-algebras}. \newblock Acta Math. \textbf{208}(1), 93-150, (2012).

\bibitem{Shlyakhtenko}
Shlyakhtenko D.: {Some estimates for non-microstates free dimension, with applications to q-semicircular families}. \newblock International Math Research Notices, \textbf{51}, 2757-2772, (2004).

\bibitem{GabeRuiz}
Gabe J., Ruiz E.: {The unital Ext-groups and classification of C*-algebras}. \newblock Glasg. Math. J., \textbf{62}(1), 202-231, (2020).



\bibitem{Zagier_positivity}
Zagier D.: {Realizability of a model in infinite statistics}. \newblock Commun. Math. Phys., \textbf{147}, 199--210, (1992).



\bibitem{MR1291240}
Jorgensen P. E. T., Werner, R. F.: {Coherent states of the {$q$}-canonical commutation relations},
\newblock Comm. Math. Phys., \textbf{164}(3) 455-471 (1994).
\newblock \urlprefix\url{http://projecteuclid.org.proxy.lib.uiowa.edu/euclid.cmp/1104270944},

\bibitem{MR1918355}
J\"{o}rgensen P. E. T., Proskurin D. P., Samo\u{\i}lenko Y.: {Generalized canonical commutation relations: representations and stability of universal enveloping {$C^*$}-algebra}, 
\newblock Symmetry in nonlinear mathematical physics, \textbf{2}, 456-460, (2002).

\bibitem{MR1893475}
Jorgensen P., Proskurin D., Samoilenko Y.: {A family of {$*$}-algebras allowing {W}ick ordering: {F}ock representations and universal enveloping {$C^*$}-algebras},
\newblock Noncommutative structures in mathematics and physics, \textbf{22}, 321-329, (2001).
              


\end{thebibliography}
\end{document}